\documentclass[11pt]{article}%
\usepackage{amssymb}
\usepackage{amsmath}
\usepackage{euler}%
\usepackage{amsfonts}%
\usepackage{graphicx}
\providecommand{\U}[1]{\protect\rule{.1in}{.1in}}
\newtheorem{theorem}{Theorem}

\newtheorem{corollary}[theorem]{Corollary}

\newtheorem{lemma}[theorem]{Lemma}

\newenvironment{proof}[1][Proof]{\noindent\textbf{#1.} }{\ \rule{0.5em}{0.5em}}
\begin{document}

\title{\textsc{uniform continuity of boundary crossing for convex sets}}
\author{\textsf{Douglas S. Bridges}}
\date{}
\maketitle

\begin{abstract}%
%TCIMACRO{\TeXButton{noindent}{\noindent}}%
%BeginExpansion
\noindent
%EndExpansion
\textsf{Let }$C$\textsf{\ be an inhabited, bounded, open, convex subset of a
Banach space, and }$\varphi$\textsf{\ the pointwise continuous mapping that
takes each point }$\left(  \xi,z\right)  $\textsf{\ of }$C\times
-C$\textsf{\ to the intersection of the segment }$\left[  \xi,z\right]
$\textsf{\ with the boundary of }$C$\textsf{. We prove constructively that
}$\varphi$\textsf{\ is }\emph{\textsf{uniformly}}\textsf{\ continuous on sets
of the form }$K\times-C$\textsf{\ with }$K$\textsf{\ well contained in }%
$C$\textsf{. }

\end{abstract}%

%TCIMACRO{\TeXButton{sf}{\normalfont\sf}}%
%BeginExpansion
\normalfont\sf
%EndExpansion%
%TCIMACRO{\TeXButton{depth 1}{\setcounter{secnumdepth}{0}}}%
%BeginExpansion
\setcounter{secnumdepth}{0}%
%EndExpansion
%

%TCIMACRO{\TeXButton{noindent}{\noindent}}%
%BeginExpansion
\noindent
%EndExpansion
This note\footnote{%
%TCIMACRO{\TeXButton{sf}{\normalfont\sf}}%
%BeginExpansion
\normalfont\sf
%EndExpansion
\textbf{MSC 2020 classification numbers:} 03F60, 46B20
\par
\ \ \thinspace\textbf{Keywords:} convex, boundary crossing, constructive}%
---written in the framework of Bishop's constructive analysis
\cite{Bishop,BB,Handbook}---deals with a substantial improvement in the
following result from pages 110--112 of \cite{BVtech}.

\begin{theorem}
\label{oct20ta1}Let $X$ be a Banach space, and $C$ an inhabited open convex
subset of $X$ that has inhabited metric complement $-C$\ and is such that
$C\cup-C$ is dense in $X$. Then there exists a pointwise continuous mapping
$\left(  \xi,z\right)  \rightsquigarrow\gamma_{\xi,z}$ of $C\times-C$ onto
$\left(  0,1\right)  $ such that for each $\xi\in C$ and each $z\in-C$,

\begin{itemize}
\item[\emph{(i)}] $0<\gamma_{\xi,z}=\inf\left\{  t\in\left[  0,1\right]
:t\xi+\left(  1-t\right)  z\in C\right\}  <1$;

\item[\emph{(ii)}] $\gamma_{\xi,z}\xi+\left(  1-\gamma_{\xi,z}\right)  z$ is
the unique point in which the segment%
\[
\left[  \xi,z\right]  \equiv\left\{  t\xi+(1-t)z:0\leq t\leq1\right\}
\]
intersects $\partial C$ (the boundary of $C$);

\item[\emph{(iii)}] if $\gamma_{\xi,z}<t\leq1$, then $t\xi+\left(  1-t\right)
z\in C$;

\item[\emph{(iv)}] if $0\leq t<\gamma_{\xi,z}$, then $t\xi+\left(  1-t\right)
z\in-C$.
\end{itemize}%

%TCIMACRO{\TeXButton{noindent}{\noindent}}%
%BeginExpansion
\noindent
%EndExpansion
Moreover, the \textsc{boundary crossing map}%
\[
\varphi:(\xi,z)\rightsquigarrow\gamma_{\xi,z}\xi+\left(  1-\gamma_{\xi
,z}\right)  z
\]
of $C\times-C$ onto $\partial C$ is pointwise continuous.
\end{theorem}%

%TCIMACRO{\TeXButton{medskip}{\medskip}}%
%BeginExpansion
\medskip
%EndExpansion
%

%TCIMACRO{\TeXButton{noindent}{\noindent}}%
%BeginExpansion
\noindent
%EndExpansion
As we shall show, adding the extra hypothesis that $C$ be bounded enables us
to prove that $\varphi$ is \emph{uniformly continuous }on certain subsets of
$C\times-C$.

Before stating our main theorem, we clarify some definitions and notation.
First, using $\rho$ to denote the standard metric on a normed space $X$, when
$x\in X$ and $S\subset X$ we write

\begin{itemize}
\item $\rho(x,S)\geq\alpha$ as shorthand for `$\left\Vert x-y\right\Vert
\geq\alpha$\emph{\ }for each\emph{\ }$y\in S$';

\item $\rho(x,S)>0$ as shorthand for `there exists $\alpha>0$ such that
$\rho(x,S)\geq\alpha$'; and

\item $\rho(x,S)<\varepsilon$ as shorthand for `there exists $y\in S$\ such
that $\left\Vert x-y\right\Vert <\varepsilon$'.
\end{itemize}%

%TCIMACRO{\TeXButton{noindent}{\noindent}}%
%BeginExpansion
\noindent
%EndExpansion
The \textsc{metric complement} and \textsc{logical complement\ }of $S$ are,
respectively,
\begin{align*}
\text{ }-S &  \equiv\left\{  x\in X:\rho(x,S)>0\right\}  \text{\ \ and}\\
\lnot S &  \equiv\left\{  x\in X:x\notin S\right\}  \text{.}%
\end{align*}
Note that if $S$ is an open set and $B(x,r)\subset\lnot S$ for some $r>0$,
then$\ x\in-S$. A set $K$ is\textsc{\ well contained}\textbf{\ }in $S$ if
there exists $r>0$ such that $x\in S$ whenever $\rho(x,K)<\varepsilon$, in
which case we write $K\subset\subset S$.%

%TCIMACRO{\TeXButton{smallskip}{\smallskip}}%
%BeginExpansion
\smallskip
%EndExpansion

Here is our main result.

\begin{theorem}
\label{oct21t1}Let $X$ be a Banach space, and $C$ an inhabited subset of $X$
that is bounded, open, convex, and such that $C\cup-C$ is dense in $X$. Let
$\varphi$ be the boundary crossing mapping of $C\times-C$ onto $\partial C$.
Then $\varphi$ is uniformly continuous on $K\times-C$ for each set
$K\subset\subset C$.
\end{theorem}%

%TCIMACRO{\TeXButton{medskip}{\medskip}}%
%BeginExpansion
\medskip
%EndExpansion

Our proof depends on three preliminaries.

\begin{lemma}
\label{oct23l1}Under the hypotheses of \emph{Theorem \ref{oct20ta1}},
let\emph{\ }$\xi\in C$, $z\in-C$, and $t<\gamma_{\xi,z}$\textbf{.
}Then\textbf{\ }$\eta\equiv t\xi+(1-t)z\in-C$.
\end{lemma}

\begin{proof}
Since $C$ and $-C$ are open, there exists $r>0$ such that \thinspace
$B(\xi,r)\subset C$ and $B(z,r)\subset-C$. Either $\left\vert t\right\vert
<r/\left\Vert \xi-z\right\Vert $, in which case $\left\Vert \eta-z\right\Vert
<r$ and therefore $\eta\in-C$, or else, as we may assume, $t\neq0$. If $t>0$,
then $\eta\in-C$ by Theorem \ref{oct20ta1}. So we are left with the case
$t<0$. In that case, let $y\in B(\eta,r)$, $w=y-\eta$, and $x=\xi+w$. Then
$x\in B(\xi,r)\subset C$ and $z+w\in B(z,r)\subset-C$. On the other hand,%
\begin{align*}
z+w &  =\frac{1}{1-t}\eta+\frac{-t}{1-t}\xi+w\\
&  =\frac{1}{1-t}(\eta+w)+\frac{-t}{1-t}(\xi+w)\\
&  =\frac{1}{1-t}y+\frac{-t}{1-t}x,
\end{align*}
so if $y\in C$ then, by convexity, $z+w\in C$, a contradiction. It follows
that $B(\eta,r)\subset\lnot C$ and therefore that $\eta\in-C$.
\end{proof}

\begin{lemma}
\label{oct23l2}Under the hypotheses of \emph{Theorem \ref{oct20ta1}}, let
$\xi\in C,z\in-C,$ and $t<s\leq1$. Suppose that $a\equiv t\xi+\left(
1-t\right)  z\in-C$ and $b\equiv s\xi+\left(  1-s\right)  z\in C$. Then
$\ t<\gamma_{\xi,z}<s$.
\end{lemma}

\begin{proof}
If $\gamma_{\xi,z}<t$, then $\gamma_{\xi,z}<t<1$, so (by Theorem
\ref{oct20ta1}(iii)) $a\in C$, which is absurd; whence $t\leq\gamma_{\xi,z}$.
Since $-C$ is open, $\rho(a,\partial C)>0$ and therefore $t<\gamma_{\xi,z} $.
On the other hand, by Lemma \ref{oct23l1}, $s\not < \gamma_{\xi,z}$ and
therefore $\gamma_{\xi,z}\leq s\leq1$; but $C$ is open, so $\rho(b,\partial
C)>0$ and therefore $\gamma_{\xi,z}<s$.
\end{proof}

\begin{lemma}
\label{oct21l1}Under the hypotheses of \emph{Theorem \ref{oct20ta1}}, let
$\xi\in C$ and $\zeta\in\partial C$. Let $\varepsilon,r,R$ be such that
$0<\varepsilon<r<R$ and $B(\xi,r)\subset C\subset B(0,R)$. Then

\begin{itemize}
\item[\emph{(i)}] $B(\tfrac{\varepsilon}{R}\xi+(1-\tfrac{\varepsilon}{R}%
)\zeta,\frac{\varepsilon r}{R})\subset C\cap B(\zeta,3\varepsilon)$ and

\item[\emph{(ii)}] $B(-\tfrac{\varepsilon}{R}\xi+(1+\tfrac{\varepsilon}%
{R})\zeta,\frac{\varepsilon r}{R})\subset-C\cap B(\zeta,3\varepsilon)$.
\end{itemize}
\end{lemma}

\begin{proof}
Note that $\zeta\in\overline{C}$, so $\left\Vert \zeta\right\Vert \leq R$.
Given $x\in B\!\,(\tfrac{\varepsilon}{R}\xi+(1-\frac{\varepsilon}{R}%
)\zeta,\frac{\varepsilon r}{R})$, we have%
\begin{align*}
\left\Vert x-\zeta\right\Vert  &  \leq\left\Vert x-\tfrac{\varepsilon}{R}%
\xi-\left(  1-\tfrac{\varepsilon}{R}\right)  \zeta\right\Vert +\left\Vert
\tfrac{\varepsilon}{R}\xi\right\Vert +\left\Vert \left(  1-\tfrac{\varepsilon
}{R}\right)  \zeta-\zeta\right\Vert \\
&  <\frac{\varepsilon r}{R}+\frac{\varepsilon}{R}\left\Vert \xi\right\Vert
+\dfrac{\varepsilon}{R}\left\Vert \zeta\right\Vert <\varepsilon+\frac
{\varepsilon}{R}R+\frac{\varepsilon}{R}R=3\varepsilon\text{.}%
\end{align*}
so $x\in B(\zeta,3\varepsilon)$. Now letting%
\[
y=\frac{R}{\varepsilon}\left(  x-\left(  1-\frac{\varepsilon}{R}\right)
\zeta\right)  ,
\]
we have
\[
\left\Vert y-\xi\right\Vert =\left\Vert \frac{R}{\varepsilon}\,\left(
x-\left(  1-\frac{\varepsilon}{R}\right)  \zeta-\frac{\varepsilon}{R}%
\xi\right)  \right\Vert <\frac{R}{\varepsilon}\frac{\varepsilon r}{R}=r
\]
and therefore $y\in B(\xi,r)\subset C$. Moreover, $0<\tfrac{\varepsilon}%
{R}<\tfrac{\varepsilon}{r}<1$, so by \cite[Cor. 11]{DCArxiv},
\[
x=\dfrac{\varepsilon}{R}y+\left(  1-\dfrac{\varepsilon}{R}\right)  \zeta\in C.
\]
Conclusion (i) now follows.

On the other hand, by Lemma \ref{oct23l1},%
\[
-\frac{\varepsilon}{R}\xi+(1+\frac{\varepsilon}{R})\zeta=-\frac{\varepsilon
}{R}\xi+(1-\frac{-\varepsilon}{R})\zeta\in-C
\]
Let\textbf{\ }$\mathsf{B}\equiv B(-\frac{\varepsilon}{R}\xi+(1+\frac
{\varepsilon}{r})\zeta,\frac{\varepsilon r}{R})$. Arguing like we did above,
we find that $\mathsf{B}\subset B(\zeta,3\varepsilon)$. Moreover, if
$x\in\mathsf{B}$, then%
\[
y\equiv\frac{R}{\varepsilon}\left(  \left(  1+\frac{\varepsilon}{R}\right)
\zeta-x\right)  \in B(\xi,r)\subset C.
\]
If also $x\in C$, then%
\[
\zeta=\frac{\varepsilon/R}{1+\varepsilon/R}y+\left(  1-\frac{\varepsilon
/R}{1+\varepsilon/R}\right)  x\in C,
\]
so $\zeta\in C\cap\partial C$, which is absurd as $C$ is open. Hence
$\mathsf{B}\subset\lnot C$; but $\mathsf{B}$ is open, so $\mathsf{B}\subset-C$
and (ii) follows.
\end{proof}%

%TCIMACRO{\TeXButton{medskip}{\medskip}}%
%BeginExpansion
\medskip
%EndExpansion

\begin{proof}
[\textbf{Proof of Theorem \ref{oct21t1}}]There exists $R>0$ such that
$C\subset B(0,R)$, and there exists $r$ with $0<r<R$ such that if
$\rho(x,K)<r$, then $x\in C$. Thus $B(\xi,r)\subset C$ for each $\xi\in K$.
Let $0<\varepsilon<r$, and consider $\xi,\xi^{\prime}\in K$ and $z,z^{\prime
}\in-C$ such that%
\[
\max\left\{  \left\Vert \xi-\xi^{\prime}\right\Vert ,\left\Vert z-z^{\prime
}\right\Vert \right\}  <\frac{\varepsilon r}{5R}\text{.}%
\]
Setting%
\[
s=\frac{\varepsilon}{R}+\left(  1-\frac{\varepsilon}{R}\right)  \gamma_{\xi
,z},
\]
we have $0<s<1$ and $1-s=(1-\frac{\varepsilon}{R})\left(  1-\gamma_{\xi
,z}\right)  $. Hence, with $\gamma$ and $\varphi$ as in Theorem \ref{oct20ta1}%
\emph{,}
\begin{align*}
& s\xi^{\prime}+\left(  1-s\right)  z^{\prime}-\frac{\varepsilon}{R}%
\xi-(1-\frac{\varepsilon}{R})\varphi(\xi,z)\\
& \ \ \ \left.  =\left(  \frac{\varepsilon}{R}+\left(  1-\frac{\varepsilon}%
{R}\right)  \gamma_{\xi,z}\right)  \xi^{\prime}+\left(  1-\frac{\varepsilon
}{R}\right)  \left(  1-\gamma_{\xi,z}\right)  z^{\prime}\right. \\
& \ \ \ \ \ \ \ \ \ \left.  -\frac{\varepsilon}{R}\xi-(1-\frac{\varepsilon}%
{R})\left[  \gamma_{\xi,z}\xi+\left(  1-\gamma_{\xi,z})z\right)  \right]
\right. \\
& \ \ \ \left.  =\frac{\varepsilon}{R}(\xi^{\prime}-\xi)+\left(
1-\frac{\varepsilon}{R}\right)  \gamma_{\xi,z}(\xi^{\prime}-\xi)\right. \\
& \ \ \ \ \ \ \ \ \ \left.  +\left(  1-\frac{\varepsilon}{R}\right)  \left(
1-\gamma_{\xi,z}\right)  (z^{\prime}-z).\right.
\end{align*}
and therefore%
\begin{align*}
& \left\Vert s\xi^{\prime}+\left(  1-s\right)  z^{\prime}-\frac{\varepsilon
}{R}\xi-(1-\frac{\varepsilon}{R})\varphi(\xi,z)\right\Vert \\
& \ \ \ \left.  \leq\frac{\varepsilon}{R}\left\Vert \xi^{\prime}%
-\xi\right\Vert +\left(  1-\frac{\varepsilon}{R}\right)  \gamma_{\xi
,z}\left\Vert \xi^{\prime}-\xi\right\Vert \right. \\
& \ \ \ \ \ \ \ \ \left.  +\left(  1-\frac{\varepsilon}{R}\right)  \left(
1-\gamma_{\xi,z}\right)  \left\Vert z^{\prime}-z\right\Vert \right. \\
& \ \ \ \left.  \leq\frac{\varepsilon}{R}\left\Vert \xi^{\prime}%
-\xi\right\Vert +\left(  1-\frac{\varepsilon}{R}\right)  \left\Vert
\xi^{\prime}-\xi\right\Vert +\left\Vert z^{\prime}-z\right\Vert \right. \\
& \ \ \ \left.  =\left\Vert \xi^{\prime}-\xi\right\Vert +\left\Vert z^{\prime
}-z\right\Vert \right. \\
& \ \ \ \left.  <\frac{\varepsilon r}{R}.\right.
\end{align*}
Thus, by Lemma \ref{oct21l1},%
\[
s\xi^{\prime}+\left(  1-s\right)  z^{\prime}\in B\left(  \frac{\varepsilon}%
{R}\xi+\left(  1-\dfrac{\varepsilon}{R}\right)  \varphi(\xi,z),\frac
{\varepsilon r}{R}\right)  \subset C\cap B(\varphi(\xi,z),3\varepsilon).
\]
%

%TCIMACRO{\TeXButton{noindent}{\noindent}}%
%BeginExpansion
\noindent
%EndExpansion
On the other hand, setting%
\[
t=\frac{-\varepsilon}{R}+\left(  1+\frac{\varepsilon}{R}\right)  \gamma
_{\xi,z}=\gamma_{\xi,z}-\left(  1-\gamma_{\xi,z}\right)  \frac{\varepsilon}%
{R}<\gamma_{\xi,z},
\]
and using computations similar to those above, we obtain:%
\begin{align*}
& \left\Vert t\xi^{\prime}+\left(  1-t\right)  z^{\prime}-\frac{-\varepsilon
}{R}\xi-(1+\frac{\varepsilon}{R})\varphi(\xi,z)\right\Vert \\
& \ \ \ \ \ \ \left.  \leq\frac{\varepsilon}{R}\left\Vert \xi^{\prime}%
-\xi\right\Vert +\left(  1+\frac{\varepsilon}{R}\right)  \gamma_{\xi
,z}\left\Vert \xi^{\prime}-\xi\right\Vert \right. \\
& \ \ \ \ \ \ \ \ \ \ \left.  +\left(  1+\frac{\varepsilon}{R}\right)  \left(
1-\gamma_{\xi,z}\right)  \left\Vert z^{\prime}-z\right\Vert \right. \\
& \ \ \ \ \ \ \left.  \leq\left\Vert \xi^{\prime}-\xi\right\Vert +\left(
1+\frac{\varepsilon}{R}\right)  \left\Vert \xi^{\prime}-\xi\right\Vert
+\left(  1+\frac{\varepsilon}{R}\right)  \left\Vert z^{\prime}-z\right\Vert
\right. \\
& \ \ \ \ \ \ \left.  \leq\left\Vert \xi^{\prime}-\xi\right\Vert +2\left\Vert
\xi^{\prime}-\xi\right\Vert +2\left\Vert z^{\prime}-z\right\Vert \right. \\
& \ \ \ \ \ \ \left.  \leq\frac{\varepsilon r}{5R}+2\frac{\varepsilon r}%
{5R}+2\frac{\varepsilon r}{5R}=\frac{\varepsilon r}{R}.\right.
\end{align*}
Thus by Lemma \ref{oct21l1}(ii),%
\[
t\xi^{\prime}+\left(  1-t\right)  z^{\prime}\in B\left(  \frac{-\varepsilon
}{R}\xi+\left(  1+\dfrac{\varepsilon}{R}\right)  \varphi(\xi,z),\frac
{\varepsilon r}{R}\right)  \subset-C\cap B(\varphi(\xi,z),3\varepsilon).
\]
It follows from Lemma \ref{oct23l2} and the convexity of $B(\varphi
(\xi,z),3\varepsilon)$ that $t<\gamma_{\xi^{\prime},z^{\prime}}<s$ and%
\[
\varphi(\xi^{\prime},z^{\prime})=\gamma_{\xi^{\prime},z^{\prime}}\xi^{\prime
}+(1-\gamma_{\xi^{\prime},z^{\prime}})z^{\prime}\in B(\varphi(\xi
,z),3\varepsilon).
\]
Hence $\left\Vert \varphi(\xi^{\prime},z^{\prime})-\varphi(\xi,z)\right\Vert
<3\varepsilon$.
\end{proof}%

%TCIMACRO{\TeXButton{medskip}{\medskip}}%
%BeginExpansion
\medskip
%EndExpansion

\begin{corollary}
\label{oct22c1}Under the hypotheses of \emph{Theorem \ref{oct21t1},} for each
$\xi\in C$ the mapping $z\rightsquigarrow\varphi(\xi,z)$ is uniformly
continuous on $-C$.
\end{corollary}%

%TCIMACRO{\TeXButton{bigskip}{\bigskip}}%
%BeginExpansion
\bigskip
%EndExpansion

%

%TCIMACRO{\TeXButton{bigskip}{\bigskip}}%
%BeginExpansion
\bigskip
%EndExpansion
%

%TCIMACRO{\TeXButton{noindent}{\noindent}}%
%BeginExpansion
\noindent
%EndExpansion
\textbf{Author's address: \ }School of Mathematics and Statistics, University
of Canterbury, Christchurch 8140, New Zealand

\vfill

\begin{flushright}
{\small dsbridges 251026}
\end{flushright}

\end{document}